\documentclass[letterpaper, 10 pt, conference]{ieeeconf}  
\IEEEoverridecommandlockouts        
\usepackage{graphicx} 
\usepackage{epstopdf} 
\usepackage{psfrag} 
\usepackage{amsmath} 
\usepackage{amssymb}  
\usepackage{amsthm,thmtools,dsfont}
\usepackage{color}
\usepackage{tabu}
\usepackage{hyperref,cleveref}
\let\labelindent\relax
\usepackage{enumitem} \setlist[itemize,1]{leftmargin=\dimexpr 22pt}
\usepackage{cite}
\usepackage{microtype,enumerate,url}
\usepackage{BOONDOX-ds}
\usepackage{tikz} 
\usepackage[fulladjust]{marginnote}
\usepackage{pgfplotstable} 
\usepackage{xstring} 
\usepackage{booktabs} 
\usepackage{colortbl}

\newcommand{\Tini}{{T_{\textup{ini}}}}
\newcommand{\Tf}{{T_{\textup{f}}}}
\newcommand{\uini}{{u_{\textup{ini}}}}

\newcommand{\yini}{{y_{\textup{ini}}}}

\usepackage{algorithm}
\usepackage{algorithmic}

\renewcommand{\theenumi}{(\roman{enumi})}

\makeatletter
\def\@IEEEsectpunct{.\ \,}
\def\paragraph{\@startsection{paragraph}{4}{\z@}{1.5ex plus 1.5ex minus 0.5ex}%
{0ex}{\normalfont\normalsize\itshape}}
\makeatother

\renewcommand*\descriptionlabel[1]{\hspace\labelsep\normalfont#1}

\declaretheorem[style=definition,qed=$\lrcorner$]{theorem}
\declaretheorem[style=definition,qed=$\lrcorner$]{corollary}
\declaretheorem[style=definition,qed=$\lrcorner$]{proposition}
\declaretheorem[style=definition,qed=$\lrcorner$]{lemma}
\declaretheorem[style=definition,qed=$\vartriangle$]{remark}

\declaretheorem[style=definition,numbered=no,qed=$\lrcorner$]{standing assumption}

\declaretheorem[style=definition]{definition}

\renewcommand\thmcontinues[1]{continued}

\newcommand {\nn}{\nonumber}
\newcommand{\beq}{\begin{equation}}
\newcommand{\eeq}{\end{equation}}
\newcommand {\bseq}{\begin{subequations}}
\newcommand {\eseq}{\end{subequations}}
\newcommand {\bma}{\left[}
\newcommand {\ema}{\right]}

\newcommand {\N}{\mathbb{N}} 	
\newcommand {\Zge}{\mathbb{Z}_{+}} 	
\newcommand {\R}{\mathbb{R}} 	
\newcommand {\Co}{\mathbb{C}} 	

\newcommand {\B}{\mathcal{B}} 	
\renewcommand {\L}{\mathcal{L}} 	

\newcommand{\closure}{\operatorname{cl}} 
\newcommand{\Image}{\operatorname{im}} 
\newcommand{\rank}{\operatorname{rank}} 
\newcommand{\transpose}{\mathsf{T}} 
\newcommand{\blockdiag}{\operatorname{block-diag}} 

\newcommand{\norm}[1]{\left\lVert#1\right\rVert}

\newcommand{\diag}{\operatorname{diag}}
\newcommand{\trace}{\operatorname{trace}}

\newcommand\Grass[2]{\operatorname{Gr}(#1,  #2) }

\newcommand\Ort[1]{\operatorname{O}(#1) }

\title{\Large {\bfseries  Behavioral uncertainty quantification for data-driven control}}

\author{Alberto Padoan,  Jeremy Coulson,   Henk J. van Waarde,  John Lygeros, and Florian D\"orfler\thanks{
		A.  Padoan,  J.  Coulson,  J.  Lygeros,  F.  D\"orfler are with 
		the Department of Information Technology and Electrical Engineering at		
		ETH Z\"urich,  Z\"urich, Switzerland 
		{\tt\footnotesize  \{apadoan,  jcoulson, lygeros, dorfler\}@control.ee.ethz.ch}. 
		H.  J.  van Waarde 
		is with the Bernoulli Institute for Mathematics,  Computer Science and Artificial Intelligence,  University of Groningen,  Groningen, The Netherlands.
		{\tt\footnotesize  h.j.van.waarde@rug.nl}. Research supported by the Swiss National Science Foundation under the NCCR Automation.
						}
}
\date{\small\today} 

\begin{document}

\maketitle
\thispagestyle{plain}
\pagestyle{plain}

\begin{abstract} 
\noindent
This paper explores the problem of uncertainty quantification in the behavioral setting for data-driven control. Building on classical ideas from robust control, the problem is regarded as that of selecting a metric which is best suited to a data-based description of uncertainties. Leveraging on Willems' fundamental lemma, restricted behaviors are viewed as subspaces of fixed dimension, which may be represented by data matrices. Consequently,  metrics between restricted behaviors are defined as distances between points on the Grassmannian, i.e., the set of all subspaces of equal dimension in a given vector space. A new metric is defined on the set of restricted behaviors as a direct finite-time counterpart of the classical gap metric. The metric is shown to capture parametric uncertainty for the class of autoregressive (AR) models. Numerical simulations illustrate the value of the metric with a data-driven  mode recognition  and control case study. 
\end{abstract}

\section{Introduction}

In a typical control design problem,  the role of data (time series) has been long dictated by \textit{indirect approaches}~\cite{ljung1999system,VanOverschee1996subspace},  where system identification is sequentially followed by model-based control.  However,  the advent of large data sets  and the ever-increasing computing power,     combined 
with the ongoing revolution brought about by machine learning technologies,  has recently triggered a renewed appreciation for \textit{direct approaches},  where the objective is to infer optimal decisions directly from measured data.   

A cornerstone of this newly emerging trend in control is a far-reaching result due to Willems and co-authors~\cite{willems2005note},  commonly known as the  \textit{fundamental lemma}.  
Leveraging on  the behavioral 
approach to 
system   
theory~\cite{willems1986time,willems1997introduction},
the fundamental lemma establishes that parametric models of a data-generating  linear time-invariant (LTI) system  
may be replaced by  raw data, 
provided the dynamics are sufficiently excited.  
Following the contributions~\cite{coulson2019data,de2019formulas,vanWaarde2020},  the number of new data-driven control algorithms has boomed over the past few years,  see, e.g.,~\cite{markovsky2021behavioral} for a recent overview.  A convincing demonstration of the potential of 
 direct approaches to 
data-driven control   
is the successful implementation of the DeePC algorithm~\cite{coulson2019data} in a wide range of experimental case studies,  including aerial robotics~\cite{coulson2021distributionally},   synchronous motor drives~\cite{carlet2020data},  grid-connected power converters~\cite{huang2019data}.

The new wave of data-driven control algorithms has primarily modeled uncertainty by ellipsoids ~\cite{coulson2021distributionally,berberich2020robust,vanWaarde2022,xue2021data,bisoffi2021data,huang2021robust}.
While effective in many circumstances,  this approach disregards the geometric structure of the data,  leading to a 
possibly coarse characterization of uncertainty.

 This   paper explores the problem  of   uncertainty quantification in data-driven control  of LTI systems. 
We seek a data-based behavioral description  of uncertainty.    
Building on the rich legacy of robust control theory~\cite{zames1981uncertainty,partington2004linear,zhou1996robust,vinnicombe2001uncertainty}, 
we identify the problem of uncertainty quantification with that of selecting a  ``natural'' metric to  study  robustness questions.      
The starting point of our analysis is a seemingly elementary,  yet profound consequence of the fundamental lemma:  restricted behaviors may be regarded as subspaces of fixed dimension and represented directly by data matrices.
Building on this premise,  we identify restricted behaviors with points on the \textit{Grassmannian} $\Grass{k}{N}$,  \textit{i.e.}, the set of all subspaces of dimension 
$k$ in $\R^N$,  endowed with the structure of  a  (quotient) manifold.   The $L$-gap metric is then introduced as a direct finite-time counterpart of the classical gap metric~\cite{zames1981uncertainty,partington2004linear,zhou1996robust,vinnicombe2001uncertainty},  
which measures the distance between graphs of input-output operators and allows one to compare the closed-loop behavior of different systems subject to 
the same feedback controller.

\textbf{Contributions}: 
The contributions of the paper are fourfold:
(i) we define a new (representation free) metric on the set of restricted behaviors; we show that the metric is easily computed via measured data and readily understood in terms of trajectories;
(ii) we show the our metric can be used for uncertainty quantification for behaviors described by AR models;
(iii) we connect the $L$-gap to the classical $\ell_2$-gap from robust control theory; and
(iv) we demonstrate the benefits brought by the $L$-gap in a data-driven mode recognition and control case study.

\textbf{Paper organization}: The remainder of this paper is organized as follows.
Section~\ref{sec:preliminaries} provides basic definitions regarding behavioral systems.
Section~\ref{sec:main_results}  introduces a new metric between restricted behaviors, 
which is then used for uncertainty quantification purposes
and shown to be closely connected to the classical $\ell_2$-gap.   
Section~\ref{sec:example} illustrates the theory with a numerical case study.
Section~\ref{sec:conclusion} provides a summary of the main results and an outlook to future research directions.

\textbf{Notation}:  
The set of positive and non-negative integers are denoted by $\N$ and $\Zge$, respectively.   
The set of positive integers $\{1, \dots , p\}$ is denoted by $\mathbf{p}$ for all ${p\in\N}$. 
The set of real numbers is denoted by $\R$.   
The transpose, image, and kernel of the matrix ${M \in \R^{p \times m}}$ are denoted by $M^{\transpose}$, $\Image M$, and $\ker M$, respectively.  
A map $f$ from $X$ to $Y$ is denoted by $f:X \to Y$; $(Y)^{X}$ denotes the set of all such maps. 
The \textit{$t$-shift} is defined as $(\sigma^t f)(t^{\prime}) = f(t+t^{\prime})$ for all ${t,t^{\prime}\in\Zge}$.  

\section{Behavioral systems}  \label{sec:preliminaries}

\subsection{Preliminaries in behavioral system theory}

Following~\cite{willems1997introduction},  we introduce some basic notions and results on behavioral systems.

\begin{definition}
A \textit{dynamical system} $\Sigma$ is a triple 
$\Sigma=(\Zge,\R^q,\B),$
where $\Zge$ is the \textit{time set}, $\R^q$ is the \textit{signal space}, and $\B \subseteq (\R^q)^{\Zge}$ is  the \textit{behavior} of the system.  
\end{definition}
\begin{definition}
A dynamical system $\Sigma=(\Zge,\R^q,\B)$ is \textit{linear} if $\B$ is a linear subspace of $(\R^q)^{\Zge}$, \textit{time invariant} if ${\sigma^t(\B) \subseteq \B}$ for all ${t \in \Zge}$, and \textit{complete} if ${\B}$ is closed in the topology of pointwise convergence.
\end{definition}
\noindent 

The structure of an LTI dynamical system is characterized by a set of integer invariants 
known as \textit{structure indices}~\cite[Section 7]{willems1986time}.
The most important ones are the \textit{number of inputs (or free variables)} $m$,  the \textit{lag} $l$, and the \textit{order} $n$.
The structure indices are intrinsic properties of a dynamical system,  as they do not depend  on  its representation.  The \textit{complexity} of a dynamical system is defined as  $c= (m,l,n)$.  The class of all complete linear,  time invariant systems (with complexity $c$)  is denoted by $\L^q$ ($\L^{q,c}$).  By a convenient abuse of notation, we shall also write $\B \in \L^{q}$ ($\B \in \L^{q,c}$).

\begin{definition}
Let $\B \in \mathcal{L}^{q}$ and  ${T\in \N}$. 
The \textit{restricted behavior (in the interval $[1,T]$)} is the set
$\B|_{T}  =\{w=\text{col}(w_1, \ldots,w_T) \in \R^{qT} \, | \, \exists \, v \in \B  \, : \, w_t = v_t, \, \forall \, t\in \mathbf{T} \}.$ 
A vector ${w \in \B|_{T}}$ is a \textit{$T$-length trajectory of the dynamical system $\B$}.
\end{definition}

The following lemma characterizes the dimension of a restricted behavior ${\B|_{L} \in \L^{q,c}}$ in terms of its complexity.

\begin{lemma}\cite[Lemma 2.1]{dorfler2022bridging} \label{lemma:subspace}
Let  ${\B \in \L^{q,c}}$. Then $\B|_{L}$ is a subspace of $\R^{qL}$,  the dimension of which is
 ${\dim \B|_{L}  =  m L+ n, }$ for ${L > l}$.  
\end{lemma}

\begin{definition}
A dynamical system $\B \in \mathcal{L}^{q}$ is \textit{controllable} if 
for every  ${T\in \N}$,   
${w^1 \in \B|_{T}}$,  and  ${w^2 \in \B}$ there exists 
${T^{\prime} \in \Zge}$, and ${w \in \B}$ such that 
${w_t = w^1_t}$ for ${t\in \mathbf{T}}$ and
${w_t = w^2_{t-T-T^{\prime}}}$ for ${t > T+T^{\prime}}$.
\end{definition}

\noindent
In other words, a dynamical system is controllable if any two trajectories can be patched together in finite time.

\subsection{The fundamental lemma}

Given a $T$-length trajectory ${w \in \R^{qT}}$ of a controllable dynamical system ${\B \in \mathcal{L}^{q}}$,  
one may obtain a non-parametric representation of the restricted behavior using 
a result first presented in~\cite{willems2005note}, which over time became known as
the \textit{fundamental lemma}. 
To state this result,  we introduce some preliminary notions.

\begin{definition}
The \textit{Hankel matrix} of depth ${L\in\mathbf{T}}$ associated with ${w \in \R^{qT}}$ is defined as
\beq \nn
\scalebox{0.95}{$
H_{L}(w) =
\bma  \nn
\begin{array}{ccccc}
w_{1} & w_{2}  & \cdots &  w_{T-L+1}   \\
w_{2} & w_{3}  & \cdots &   w_{T-L+2}   \\
\vdots  & \vdots  & \ddots & \vdots  \\
w_{L} & w_{L+1}  & \cdots  & w_{T}
\end{array}
\ema \in \R^{(qL) \times (T-L+1)} .
$}
\eeq
\end{definition}

\begin{definition}
A vector ${u \in\R^{mT}}$ is \textit{persistently exciting of order $L$} if $H_L(u)$ is full row rank, i.e.,
${\rank H_L(u) = mL}$.
\end{definition}

\noindent
Persistency of excitation plays a key role in system identification and adaptive control~\cite{ljung1999system, VanOverschee1996subspace,astrom1995adaptive}.  A necessary condition for ${u \in \R^{mT}}$ to be 
persistently exciting of order $L$ is that $H_L(u)$ has at least as many columns as rows, i.e., ${T \ge T_{\min} =  (m+1)L-1}$. We are now ready to state the fundamental lemma~\cite{willems2005note}.
\begin{lemma}[Fundamental lemma] \label{lemma:fundamental_lemma}
Consider a controllable dynamical system ${\B \in \mathcal{L}^{q}}$,  with input/output partition ${w = (u,y)}$.  
Assume ${w^d = (u^d,y^d)\in \B|_{T}}$ and $u^d$ is persistently exciting of order $L + n$.
Then 
${\B|_{L} = \Image H_L(w^d).}$
\end{lemma}

\noindent
Lemma~\ref{lemma:fundamental_lemma} is of paramount importance in data-driven control~\cite{markovsky2020identifiability}. 
It provides conditions for the restricted behavior $\B|_{L}$ to be completely characterized by the image of the Hankel matrix $H_L(w^d)$.
As a result,  the subspace ${\Image H_L(w^d)}$ can be regarded as a non-parametric representation of the dynamical system $\B$,  so long as  $L$-length trajectories are considered. The controllability and persistency of excitation assumptions can be removed by focusing on behaviors of fixed complexity and using the rank condition
\beq  \label{eq:generalized_persistency_of_excitation}
\rank H_L(w^d)   =  mL+ n.
\eeq

\begin{lemma}~\cite[Corollary 19]{markovsky2020identifiability} \label{lemma:fundamental_generalized}
Consider a dynamical system ${\B \in \mathcal{L}^{q,c}}$ and an associated $T$-length trajectory ${w^d \in \B|_{T}}$.
For ${L>l}$,   ${\B|_{L} = \Image H_L(w^d)}$
if and only if~\eqref{eq:generalized_persistency_of_excitation} holds.
\end{lemma}

For convenience,  in the sequel a $T$-length trajectory ${w^d}$ of a dynamical system ${\B \in \mathcal{L}^{q,c}}$ is said to be     \textit{sufficiently excited of order ${L}$}    if it satisfies the rank condition~\eqref{eq:generalized_persistency_of_excitation}. All of these ``low rank'' results hold obviously for the \emph{deterministic} LTI case, but they can also be used to design effective de-noising schemes by low-rank approximation~\cite{markovsky2021behavioral}.

\section{A metric on restricted behaviors} \label{sec:main_results}

This section explores the issue of uncertainty quantification using a data-based behavioral description of uncertainties. 
The starting point of our analysis is a seemingly elementary,  yet profound consequence of the fundamental lemma: 
restricted behaviors may be regarded as subspaces of equal dimension,  which may be represented directly by data matrices.
Thus,  restricted behaviors may be identified with points on the \textit{Grassmannian}   $\Grass{k}{N}$,  \textit{i.e.}, the set of all subspaces of dimension 
$k$ in $\R^N$,    endowed with the structure of  a (quotient) manifold~\cite[p.63]{boothby2003introduction}.   
Metrics between restricted behaviors thus arise from the underlying Grassmannian structure.

\begin{proposition} \label{thm:metrics_on_behaviors}
The function $d$ is a metric on the set of all restricted behaviors $\B|_L \in \mathcal{L}^{q,c}$, with ${L>l}$, 
whenever $d$ is a metric on $\Grass{mL+n}{qL}$.
\end{proposition}

\begin{proof}
Let ${L>l}$ and let $d$ be a metric on $\Grass{mL+n}{qL}$.  By Lemma~\ref{lemma:subspace},  the set of all restricted behaviors ${\B|_L \in \mathcal{L}^{q,c}}$ is a subset of $\Grass{mL+n}{qL}$.  Then the set of all restricted behaviors $\B|_L \in \mathcal{L}^{q,c}$ endowed with the metric $d$ is also a metric space,  since any subset of a metric space is itself a metric space with respect to the induced metric~\cite[p.38]{carothers2000real}.
\end{proof}

With these premises,   a natural question is:  what is a good notion of distance for restricted behaviors? 
Ideally,  a metric should be intrinsic,   
easily computed,
and 
readily understood in system-theoretic terms.
The aforementioned properties provide an identikit of the desired distance and pave the way for the discussion in this section,  where we explore a notion of distance between restricted behaviors.

\subsection{The gap between restricted behaviors}

The gap metric plays a pivotal role in control  theory\cite{zames1981uncertainty,partington2004linear,zhou1996robust,vinnicombe2001uncertainty}
and,  in many ways, it reflects the intuitive notion of distance between subspaces.
This section introduces the $L$-gap metric as a direct finite-time counterpart of the classical gap metric.   To this end,  we recall a few preliminary notions.

Let ${\mathcal{S}}$ be a normed space with norm $\norm{\,\cdot\,}$.  Let $v\in \mathcal{S}$ and let $\mathcal{W}$ be subspace of $\mathcal{S}$. 
The \textit{distance between $v$ and $\mathcal{W}$} is defined as~\cite[p.7]{kato1980perturbation} 
\beq \nn
\delta(v,\mathcal{W}) = \inf_{w \in\mathcal{W}} \norm{v-w}.
\eeq
\noindent
If $\norm{\,\cdot\,}_2$ is the Euclidean $2$-norm in $\R^N$,  the distance between  $v$ and $\mathcal{W}$ is the distance between $v$ and its projection onto $\mathcal{W}$,  \textit{i.e.},  ${\delta(v,\mathcal{W})= \norm{(I-P_\mathcal{W})v}_2,}$ where $P_\mathcal{W}$ is the orthogonal projector onto the subspace $\mathcal{W}$.

\begin{definition}  \cite[p.30]{partington2004linear}
Let $\mathcal{V}$ and $\mathcal{W}$ be closed subspaces of a Hilbert space $\mathcal{H}$.    
The \textit{gap between  $\mathcal{V}$ and $\mathcal{W}$} is defined as
\beq \label{eq:gap}
\text{gap}_{\mathcal{H}}(\mathcal{V},\mathcal{W}) 
=
 \norm{P_{\mathcal{V}} - P_{\mathcal{W}}},
\eeq
where $P_\mathcal{V}$ and $P_\mathcal{W}$ are the orthogonal projectors onto $\mathcal{V}$  and $\mathcal{W}$,  respectively. 
\end{definition}

\noindent
The gap between  $\mathcal{V}$ and $\mathcal{W}$ may be expressed as~\cite[p.30]{partington2004linear}
\beq\nn
\text{gap}_\mathcal{H}(\mathcal{V},\mathcal{W}) 
	= \max\left\{   \norm{(I-P_{\mathcal{W}})P_{\mathcal{V}}} 	, \norm{(I-P_{\mathcal{V}})P_{\mathcal{W}}} 		\right\} .
\eeq
In particular,  ${0 \le \text{gap}_{\mathcal{H}}(\mathcal{V},\mathcal{W}) \le 1}$ for all  $\mathcal{V}$ and $\mathcal{W}$. To streamline the exposition,   we also recall the notion of \textit{directed gap between  $\mathcal{V}$ and $\mathcal{W}$}  which is defined as
\beq \label{eq:directed_gap}
\overset{\rightharpoonup}{\text{gap}}_{\mathcal{H}}(\mathcal{V},\mathcal{W}) 
= 
\sup_{v \in \mathcal{V} \atop \norm{v}=1}\delta(v,\mathcal{W}) =  \norm{(I-P_{\mathcal{W}})P_{\mathcal{V}}}.
\eeq 
Clearly, 
$\text{gap}_{\mathcal{H}}(\mathcal{V},\mathcal{W}) 
	= \max\{ \,
			\overset{\rightharpoonup}{\text{gap}}_{\mathcal{H}}(\mathcal{V},\mathcal{W})  ,\, \,
			\overset{\rightharpoonup}{\text{gap}}_{\mathcal{H}}(\mathcal{W},\mathcal{V})  \,
\}$.
Note that no explicit mention to any particular choice of $\norm{\,\cdot\,}$ is actually needed when the ambient Hilbert space $\mathcal{H}$ is $\R^N$,  since all the gap functions are equivalent~\cite[p.91]{stewart1990perturbation}.   Throughout the paper,  we denote  by $\text{gap}$ the gap metric corresponding  the Euclidean $2$-norm $\norm{\,\cdot\,}_2$ to streamline the notation.

We are now ready to introduce a notion of distance between restricted behaviors. 

\begin{definition}  \label{def:L_gap}
Let ${\B \in \mathcal{L}^{q,c}}$ and ${\tilde{\B} \in \mathcal{L}^{q,c}}$.  For ${L\in\Zge}$,  the \textit{$L$-gap between  $\B$ and $\tilde{\B}$} is defined as 
\beq \label{eq:gapL_def}
\text{gap}_{L}(\B,\tilde{\B})  
 = \text{gap}(\B|_{L},\tilde{\B}|_{L})   .
\eeq
The \textit{directed $L$-gap between  $\B$ and $\tilde{\B}$} is defined as
$\overset{\rightharpoonup}{\text{gap}}_L(\B,\tilde{\B})   
=  
\overset{\rightharpoonup}{\text{gap}}(\B|_{L},\tilde{\B}|_{L}). $
\end{definition}

\begin{remark}
The $L$-gap can also be defined for behaviors with different lags.  However, for clarity of exposition we define it here for behaviors of the same complexity $c$.
\end{remark}

\noindent  
By Proposition~\ref{thm:metrics_on_behaviors} and since $\text{gap}$ is a metric on $\Grass{k}{N}$  for ${k,  N\in\N}$~\cite[p.93]{stewart1990perturbation}, 
we have the following result.  

\begin{corollary}
The set of all restricted behaviors $\B|_L \in \mathcal{L}^{q,c}$, with ${L>l}$,   equipped with   $\text{gap}_{L}$ is a metric space.
\end{corollary}

\begin{remark}[Geometry]
The gap metric has a well-known geometric interpretation in terms of the sine of the largest principal angle between two subspaces. 
In particular,  as an immediate consequence of~\cite[Theorem 4.5]{stewart1990perturbation},  for ${L>l}$,  the $L$-gap between 
${\B\in \mathcal{L}^{q,c}}$ and ${\tilde{\B}\in \mathcal{L}^{q,c}}$ is 
${\text{gap}_{L}(\B,\tilde{\B})   = \sin \theta_{\max}, }$ where ${\theta_{\max}}$ is the largest principal angle between the subspaces $\B|_{L}$ and $\tilde{\B}|_{L}$ (see Appendix~\ref{ssec:appendix-angles} for more detail on principal angles).
\end{remark}

\begin{remark}[Data-based computation]
The $L$-gap between behaviors can be directly computed from the knowledge of sufficiently excited trajectories.
Let ${w^{d} \in \B|_{T}}$ and ${\tilde{w}^{d} \in \tilde{\B}|_{T}}$ be sufficiently excited $T$-length trajectories of order ${L}$,  with  ${L>l}$.  
Let
\[
\begin{aligned}
&H_L(w^d) = [\,U_{1} \, U_{2}\,]  \begin{bmatrix} S & 0 \\ 0 & 0\end{bmatrix} \begin{bmatrix}V_1 \\ V_2\end{bmatrix},\\
&H_L(\tilde{w}^d) = [\,\tilde{U}_{1} \, \tilde{U}_{2}\,]  \begin{bmatrix}\tilde{S} & 0 \\ 0 & 0\end{bmatrix} \begin{bmatrix}\tilde{V}_1 \\ \tilde{V}_2\end{bmatrix}
\end{aligned}
\]
be the singular value decomposition (SVD) of the Hankel matrices ${H_L(w^d)}$ and ${H_L(\tilde{w}^d)}$  
with ${U_{1} \in \R^{qL \times (mL+n)}}$ and ${\tilde{U}_{1} \in \R^{qL \times (mL+n)}}$, respectively.
Then
\begin{equation}\label{eq:U1U2}
{\text{gap}_{L}(\B,\tilde{\B}) = \lVert{U_{1} U_{1}^{\transpose}  - \tilde{U}_{1} \tilde{U}_{1}^{\transpose}}\rVert_2} = \|\tilde{U}_2^{\transpose}U_1\|_2
\end{equation}
where the first equality comes from the fact that
$\text{gap}_{L}(\B,\tilde{\B})   =  \lVert{P_{\B|_{L}} - P_{\tilde{\B}|_{L}}\rVert}_2$
and since $P_{{\B}|_{L}} = U_{1} U_{1}^{\transpose}$ and $P_{\tilde{\B}|_{L}}= \tilde{U}_{1} \tilde{U}_{1}^{\transpose}$~\cite[p.82]{golub2013matrix}. The second identity is due to~\cite[Thm 2.5.1]{golub2013matrix}.
\end{remark}

\begin{remark}[Interpretation in terms of trajectories]
Given ${\B \in \mathcal{L}^{q,c}}$,  consider the problem of estimating the closest trajectory ${w \in \B|_{L}}$ to a given measured 
trajectory ${\tilde{w} \in \R^{qL} }$ which belongs to a possibly distinct behavior ${\tilde{\B} \in \mathcal{L}^{q,c}}$, \textit{i.e.},
\beq \nn \label{eq:initial_condition_estimation}
\begin{array}{ll}
    \displaystyle \underset{w \in  \B|_{L}} {\mbox{minimize}}    & \Vert w - \tilde{w} \Vert_{2}^2 , \\
    \mbox{subject to}    & \tilde{w} \in  \tilde{\B}|_{L} .
\end{array} 
\eeq
By Lemma~\ref{lemma:subspace},  $\B|_{L}$ is a subspace and the estimation error is
\beq \label{eq:initial_condition_estimation_error}
\inf_{w \in \B|_{L}}  \norm{w - \tilde{w}}_2   = \norm{(I-P_{{\B}|_{L}})\tilde{w}}_2 .
\eeq
Now suppose  ${\tilde{\B}}$ is known to be such that ${\text{gap}_L(\B,\tilde{\B}) \le \epsilon}$.  Then 
\beq \nn \label{eq:initial_condition_estimation_error_worst_case}
\sup_{\tilde{w} \in \tilde{\B}|_{L} \atop \norm{\tilde{w}}_2 \not = 0 } \inf_{w \in \B|_{L}}
\frac{ \norm{w - \tilde{w}}_2}{\norm{\tilde{w}}_2}  
\le 
 \epsilon .
\eeq
In other words,  $\text{gap}_L(\B,\tilde{\B})$ is an upper bound for the \textit{worst case} relative estimation error.   
The domain of the $L$-gap metric may be extended 
to measure distances between subspaces of different dimension~\cite{ye2016schubert},  so these results may be used,  e.g.,  for smoothing of a noisy trajectory $\tilde{w}$.  We elaborate more on this in Section~\ref{sec:conclusion}.   
\end{remark}

\subsection{Uncertainty quantification}

The following theorem, which is inspired by \cite[Prop. 7]{qui1992pointwise}, provides an upper and a lower bound on the $L$-gap in case the restricted behaviors have a specific form.

\begin{theorem}
\label{t:gapandtwonorm}
Let ${\B\in \mathcal{L}^{q,c}}$ and ${\tilde{\B}\in \mathcal{L}^{q,c}}$.    Given ${L \in \Zge}$,  with  ${L>l}$,   assume
\begin{equation}
\label{behaviorsIF}
\B|_L = \Image 
\bma
\begin{array}{c}
I \\
F
\end{array}
\ema
\text{ and }
\tilde{\B}|_L = \Image 
\bma
\begin{array}{c}
I \\
\tilde{F}
\end{array}
\ema.
\end{equation}
Then 
\begin{equation}
\label{boundsgapL}
\frac{  \lVert F - \tilde{F} \rVert_2 }{\sqrt{1+\norm{F}_2^2} \sqrt{1+\lVert \tilde{F} \rVert_2^2}} \le
\text{gap}_L (\B, \tilde{\B})
 \le  \lVert F - \tilde{F} \rVert_2.
\end{equation}
\end{theorem}

\begin{proof}
We first prove the left inequality, i.e., the lower bound on the $L$-gap. By~\eqref{eq:U1U2}, 
\begin{align*}
\text{gap}_L (\B, \tilde{\B}) \!=\! \lVert (I\!+\!\tilde{F}\tilde{F}^\top)^{-\frac{1}{2}} \! \begin{bmatrix}
-\tilde{F} & I
\end{bmatrix} \! \begin{bmatrix}
I \\ F
\end{bmatrix} \!(I\!+\!F^\top F)^{-\frac{1}{2}} \rVert_2 \\
\geq \sigma_{\text{min}}((I+\tilde{F}\tilde{F}^\top)^{-\frac{1}{2}}) \: \sigma_{\text{min}}((I+F^\top F)^{-\frac{1}{2}}) \: \lVert F-\tilde{F}\rVert_2,
\end{align*}
where $\sigma_{\text{min}}(X)$ denotes the smallest singular value of a given matrix $X$. Next, we will work out $\sigma_{\text{min}}((I+\tilde{F}\tilde{F}^\top)^{-\frac{1}{2}})$. Note that
\begin{align*}
\sigma_{\text{min}}((I+\tilde{F}\tilde{F}^\top)^{-\frac{1}{2}}) &= \frac{1}{\sigma_{\text{max}}((I+\tilde{F}\tilde{F}^\top)^{\frac{1}{2}})} \\
&= \frac{1}{\sqrt{ \sigma_{\text{max}}(I+\tilde{F}\tilde{F}^\top) }} \\
&= \frac{1}{\sqrt{ \lVert(I+\tilde{F}\tilde{F}^\top)\rVert_2 }},
\end{align*}
where $\sigma_{\text{max}}(X)$ denotes the largest singular value of $X$. Finally, note that for any eigenvalue $\lambda$ and corresponding eigenvector $v$ of $\tilde{F}\tilde{F}^\top$, we have that
$$
(I+\tilde{F}\tilde{F}^\top)(I+\tilde{F}\tilde{F}^\top)v = (1+\lambda)^2 v.
$$
Therefore, every singular value of $(I+\tilde{F}\tilde{F}^\top)$ is of the form $1+\sigma$, where $\sqrt{\sigma}$ is a singular value of $\tilde{F}$. We conclude that $\lVert I+\tilde{F}\tilde{F}^\top \rVert_2 = 1 + \lVert \tilde{F} \rVert_2^2$. This establishes that
$$
\sigma_{\text{min}}((I+\tilde{F}\tilde{F}^\top)^{-\frac{1}{2}}) = \frac{1}{ \sqrt{1 + \lVert \tilde{F} \rVert_2^2 }}.
$$
In similar fashion, we can prove that 
$$
\sigma_{\text{min}}((I+F^\top F)^{-\frac{1}{2}}) = \frac{1}{\sqrt{ 1 + \lVert F \rVert_2^2 }}.
$$
By substituting the latter two equalities in the lower bound on $\text{gap}_L (\B, \tilde{\B})$ we obtain the left inequality of \eqref{boundsgapL}.

To prove the upper bound on the gap, i.e., the right inequality of \eqref{boundsgapL}, we again use \eqref{eq:U1U2} to obtain
\begin{align*}
\text{gap}_L (\B, \tilde{\B}) \!&\leq\! \lVert (I\!+\!\tilde{F}\tilde{F}^\top)^{-\frac{1}{2}} \rVert_2 \: \lVert F-\tilde{F} \rVert_2 \: \lVert(I\!+\!F^\top F)^{-\frac{1}{2}}\rVert_2 \\
&\leq \lVert F-\tilde{F} \rVert_2
\end{align*}
since both $\lVert (I\!+\!\tilde{F}\tilde{F}^\top)^{-\frac{1}{2}} \rVert_2 \leq 1$ and $\lVert(I\!+\!F^\top F)^{-\frac{1}{2}}\rVert_2 \leq 1$. This establishes the upper bound in \eqref{boundsgapL}, which proves the theorem. 
\end{proof}

We have the following corollary for AR models.
\begin{corollary}\label{cor:ARX}
Let ${\B \in \mathcal{L}^{q}}$ and ${\tilde{\B} \in \mathcal{L}^{q,c}}$.  
Given ${L \in \Zge}$,  with  ${L>l}$,   assume $\B|_{L}$ and $\tilde{\B}|_{L}$
are defined by the single-input, single-output AR models with real valued coefficients $\{a_k,b_k,\tilde{a}_k,\tilde{b}_k\}_{k=0}^{L-1}$, respectively:
\begin{equation}\label{eq:ARX}
\begin{aligned}
y_{t+L-1} &= \sum_{k=0}^{L-2}a_ky_{t+k} +\sum_{k=0}^{L-1} b_ku_{t+k} , \\
y_{t+L-1} &= \sum_{k=0}^{L-2}\tilde{a}_ky_{t+k} +\sum_{k=0}^{L-1} \tilde{b}_ku_{t+k} .
\end{aligned}
\end{equation}
Assume 
\begin{align*}
F &= 
\bma 
\begin{array}{ccccccccc}
 a_0  & b_0 & \dots & a_{L-2} & b_{L-2} & b_{ L-1}
\end{array}
 \ema , \\
\tilde{F} &= 
\bma 
\begin{array}{ccccccccc}
\tilde{a}_0  & \tilde{b}_0 & \dots & \tilde{a}_{L-2} & \tilde{b}_{L-2} & \tilde{b}_{ L-1}
\end{array}
\ema .
\end{align*}
are such that $\|F - \tilde{F}\|_2 \leq\epsilon$. Then 
$\textup{gap}_L(\B,\tilde{\B})\leq \epsilon$.
\end{corollary}
\begin{proof}
Note that $\B|_{L}=\textup{ker}\begin{bmatrix}F & -I\end{bmatrix}.$
Indeed, by definition of $F$, given any trajectory 
$$w=\textup{col}(y_0,u_0,\dots,y_{L-2},u_{L-2},u_{L-1},y_{L-1}) \in\B|_{L}, $$
we have $\begin{bmatrix}F & -I\end{bmatrix}w=0.$
Note that 
$\textup{ker}\begin{bmatrix}F & -I\end{bmatrix}=\textup{im}\begin{bmatrix}I\\ F\end{bmatrix}.$
The same can be shown for $\tilde{F}$.  Leveraging on Theorem~\ref{t:gapandtwonorm} yields the desired result.
\end{proof}

\begin{remark}
We presented Corollary~\ref{cor:ARX} for single-input single-output models for clarity of exposition, but the result holds for more general multi-input multi-output systems. Corollary~\ref{cor:ARX} raises the natural open question of how to relate the $L$-gap to classical uncertainty models~\cite{zames1981uncertainty,partington2004linear,zhou1996robust,vinnicombe2001uncertainty} including additive, multiplicative, and coprime factor uncertainties. This is left as an area of future work.
\end{remark}

\subsection{Connection with the $\ell_2$-gap metric}

The gap metric plays a central role in robust control theory~\cite{zames1981uncertainty,partington2004linear,zhou1996robust,vinnicombe2001uncertainty},  where finite-dimensional,  LTI systems are regarded as operators acting on a 
 given Hilbert space $\mathcal{H}$ (such as\footnote{$\ell_2^m$ is the Hilbert space of square summable sequences $u: \Zge \to \R^m$,  with norm  
$${\norm{u}_{\ell_2}  = \sum_{t=0}^{\infty} \norm{u_t}_2^2}.$$  
$H_2(\mathbb{D}^c)^m$ is the Hardy space of functions ${f :\Co \to \Co^m}$ which are analytic in the complement  of the closed unit disk $\mathbb{D}$,  with norm~\cite[p.13]{partington2004linear}
\beq \nn
\norm{f}_{H_2(\mathbb{D}^c)} = \sup_{|r|>1} \left( \frac{1}{2\pi}  \int_{0}^{2\pi} \norm{f(re^{i\theta})}_2^2 d\theta \right)^{1/2} .
\eeq
} 
$\ell_2$ or $H_2(\mathbb{D}^c)$).  In this context,  the distance between finite-dimensional,    LTI  systems $\Sigma$ and $\tilde{\Sigma}$ is defined in terms of the    \textit{gap}   between the graphs of the corresponding input-output operators. 
 
\begin{definition}~\cite[p.30]{partington2004linear} \label{def:gap_l2}
Let\footnote{Given a mapping ${P: \mathcal{U} \to \mathcal{Y}}$,  then its \textit{graph} is the subset
$\text{graph}(P)$ of ${\mathcal{U} \times \mathcal{Y}}$ defined as 
${\text{graph}(P) = \{(u,Pu) \, :\,  u \in \mathcal{U} \}}$~\cite[p.17]{partington2004linear}. 
If $P$ is an operator defined on a normed subspace of  $\mathcal{U}$,     then that subspace is called the \textit{domain} of the operator $P$ and is denoted by ${\text{dom}(P)}$~\cite[p.18]{partington2004linear}.   To simplify the exposition,   if ${\mathcal{U}=\ell_2^m}$ and  ${\mathcal{Y}=  \ell_2^p}$,  we write
$\text{dom}_{\ell_2}(P)
=
\left\{ 
u \in   \ell_2^{m} 
\, : \,
P u
\in \ell_2^{p} 
\right\}$
and
$\text{graph}_{\ell_2}(P) = 
\{ 
(u,Pu)  \in \ell_2^{m+p} 
\, : \,
u \in \text{dom}_{\ell_2}(P)\} 
$. }  
${\mathcal{H} = \mathcal{U}\times \mathcal{Y}}$,  with $\mathcal{U}$ and $\mathcal{Y}$ Hilbert spaces.
Let ${P: \text{dom}(P)  \to \mathcal{Y}}$ and  ${\tilde{P}: \text{dom}(\tilde{P})  \to \mathcal{Y}}$ be closed operators,  
with ${\text{dom}(P)}$ and ${\text{dom}(\tilde{P})}$  being   subspaces of $\mathcal{U}$.
The \textit{gap between $P$ and $\tilde{P}$} is defined as
\beq \nn 
\text{gap}_{\mathcal{H}}(P,\tilde{P}) = \text{gap}_{\mathcal{H}}(\text{graph}(P),\text{graph}(\tilde{P})) .
\eeq
\end{definition}

We now show that,  under certain assumptions,  the $L$-gap metric can be connected to the classical $\ell_2$-gap metric.  
The proof is deferred to Appendix~\ref{ssec:proof_thm_2}.

\begin{theorem} \label{thm:gap_connection}
Let ${\B \in \mathcal{L}^{q}}$ and ${\tilde{\B} \in \mathcal{L}^{q,c}}$.  Assume $\B = \text{graph}_{\ell_2} (P)$ and $\B = \text{graph}_{\ell_2} (\tilde{P})$, 
with ${P}$ and ${\tilde{P}}$ bounded linear operators on $\ell_2^m$.  Then 
\beq \nn
\lim_{L \to \infty} \text{gap}_L(\B,\tilde{\B}) \le  \text{gap}_{\ell_2}(P,\tilde{P}) .
\eeq 
\end{theorem}

\begin{remark} 
The problem of uncertainty quantification has a long history in robust control~\cite{zames1981uncertainty,partington2004linear,zhou1996robust,vinnicombe2001uncertainty},
where systems are classically defined over an infinite time horizon.  The comparison between the $\ell_2$-gap and the $L$-gap thus requires ${L\to\infty}$. 
We observe that a   ``data-driven gap metric'' is defined and linked to the $H_2(\mathbb{D}^c)$-gap metric~\cite[Lemma 6]{koenings2017data}. Theorem~\ref{thm:gap_connection} provides a representation free version of the argument given in~\cite[Lemma 6]{koenings2017data}. 
The connection between the $L$-gap metric and the $H_2(\mathbb{D}^c)$-gap metric is left as a question for future research. 
\end{remark}

\section{Application: Mode Recognition and Control} \label{sec:example}

We envision many applications of the $L$-gap metric, e.g., prediction error quantification,  robustification in data-driven control,  and fault detection and isolation.   In the following case study,  we use it as an analysis tool in the spirit of mode recognition and control.  Namely, we determine the mode of a switched autoregressive exogenous (SARX) system directly from data for the purpose of data-driven control. 

Consider a SARX system~\cite{du2018robust} with 2 modes given by
\begin{equation}
\label{eq:SARX}
\begin{aligned}
&y_t = 0.2y_{t-1} + 0.24 y_{t-2} + 2u_{t-1} + n_t ,   \\
&y_t = 0.7y_{t-1} -0.12 y_{t-2} + 1u_{t-1} + n_t ,
\end{aligned}
\end{equation}
where $u_t\in\R$ and $y_t\in\R$ are the inputs and outputs at time $t\in\Zge$, and $n_t\sim\mathcal{N}(0,\sigma^2)$ is observation noise with $\sigma = 10^{-4}$ and truncated to the interval $[-3\sigma,3\sigma]$. 

We consider the problem of performing data-driven control~\cite{coulson2019data},  while recognizing switches in the system's mode. To this end, we use DeePC~\cite{coulson2019data} which solves the following optimal control problem in a receding horizon fashion for some data matrix $D$ serving as a predictive model for allowable trajectories:
\begin{equation}\label{eq:deepc}
\begin{aligned}
\underset{y,u,g}{\textup{minimize}}\quad
&2000\|y-r\|^2+\|u\|^2 + 20\|g\|^2\\
\text{subject to\quad}
&Dg=
\textup{col}(\uini,u,\yini,y),
\end{aligned}
\end{equation}
where $(\uini,\yini)$ is the most recent $\Tini$-length trajectory of the system (used to implicitly fix the initial condition from which the $\Tf$-length prediction, $(u,y)$ evolves), and $r\in\R^{\Tf}$ is a given reference trajectory. We select $\Tini=2$ and $\Tf=5$.

By Lemma~\ref{lemma:fundamental_generalized}, any data matrix $D$ containing sufficiently exciting data from a particular system mode describes (approximately due to noise) the subspace in which trajectories live for that mode. By performing an SVD of $D$, we can identify a large decrease in the singular values indicating the dimension of the subspace of allowable trajectories. Note that SVD does not necessarily preserve structure, but we only require a basis for the restricted behavior. In this case, the subspace dimension is given by~\eqref{eq:generalized_persistency_of_excitation} and is equal to $\Tini+\Tf+n=9$ (see Fig~\ref{fig:singular_values_offline}). Distinguishing the modes would not be possible by looking at the singular values of the data matrices alone. We propose the use of the $L$-gap to distinguish the modes.

\begin{figure}[h]
\includegraphics[width=1\linewidth]{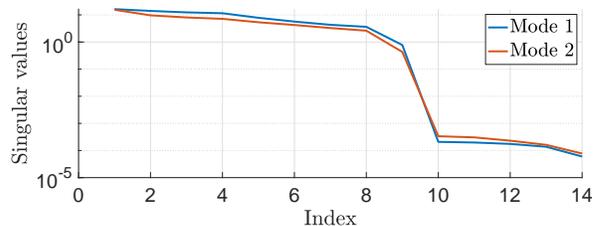}
\caption{Singular values of data matrices representing restricted behaviors of each mode of system~\eqref{eq:SARX}.}
\label{fig:singular_values_offline}
\end{figure}

We propose the following data-driven mode recognition and control strategy. Before starting, let $t\geq 0$ denote the current time, and fix the matrix $D$ in~\eqref{eq:deepc}. The first step is to compute an SVD of $D$ and form a basis $D_{\textup{basis}}$ using the first $\Tini+\Tf+n$ left singular vectors. Next compute an SVD of a matrix with $M$ columns containing the most recent $\Tini+\Tf$-length trajectories, denoted $H_t$, and form a basis, denoted $H_{t,\textup{basis}}$, using the first $\Tini+\Tf+n$ left singular vectors. Fix a threshold $\epsilon>0$. If $\textup{gap}_{\Tini+\Tf}(\Image H_{t,\textup{basis}},\Image D_{\textup{basis}})>\epsilon$, set $D=H_t$. This can be thought of as adopting the most recent data as the predictive model in~\eqref{eq:deepc} only when the gap between the predictive model $D$ and the most recent data $H_t$ is larger than some pre-defined threshold. Equipped with the data matrix $D$, solve~\eqref{eq:deepc} for the optimal predicted input trajectory $(u_1^\star,\dots,u_{\Tf}^\star)$ and apply $u_t = u_1^{\star}$ to the system. Measure $y_t$ and set $(\uini,\yini)$ in~\eqref{eq:deepc} to the most recent $\Tini$-length trajectory of the system. Update $H_t$ by deleting the first column and adding the most recent $\Tini+\Tf$-length trajectory as the last column. This process is repeated in order to perform simultaneous data-driven mode recognition and control.

The strategy above has been simulated with $\epsilon=0.3$ on system~\eqref{eq:SARX} for $t\in[0,70]$. We arbitrarily initialize the predictive model $D$ in~\eqref{eq:deepc} to be a matrix containing sufficiently exciting data from mode 1. However, the system starts in mode 2 and only switches to mode 1 at ${t=40}$. The strategy is compared to data-driven control without mode recognition, i.e., where $D$ is kept constant in~\eqref{eq:deepc}.
The results are shown in Figures~\ref{fig:control_traj} and~\ref{fig:gap}.  We observe in Figure~\ref{fig:control_traj} that the controlled output trajectory is offset from the desired reference. This is due to the fact that we are using the wrong data set in~\eqref{eq:deepc} for predicting optimal trajectories. However, the $L$-gap between $D_{\textup{basis}}$ and $H_{t,\textup{basis}}$ quickly increases above the threshold $\epsilon$, thus successfully recognizing a discrepancy between the current mode of the system and the data being used for control (see Figure~\ref{fig:gap}). During this transient phase, the moving window contains a mixture of data containing trajectories from mode 2, and mode 1. However, at approximately ${t=22}$, the $L$-gap successfully recognizes that the data matrix $D$ used in~\eqref{eq:deepc} is consistent with the current mode of the system (mode 2). The control performance after this transient phase then improves. This is again illustrated during the mode switch at ${t=40}$. On the other hand, the data-driven control strategy with no mode recognition does not adapt to mode switches and has poor performance until ${t=40}$ where the system switches incidentally to mode 1 thus matching with the fixed data matrix $D$ being used in this strategy. This case study suggests that the $L$-gap is a suitable tool for data-driven online mode recognition and control.

\begin{figure}[t!]
\includegraphics[width=1\linewidth]{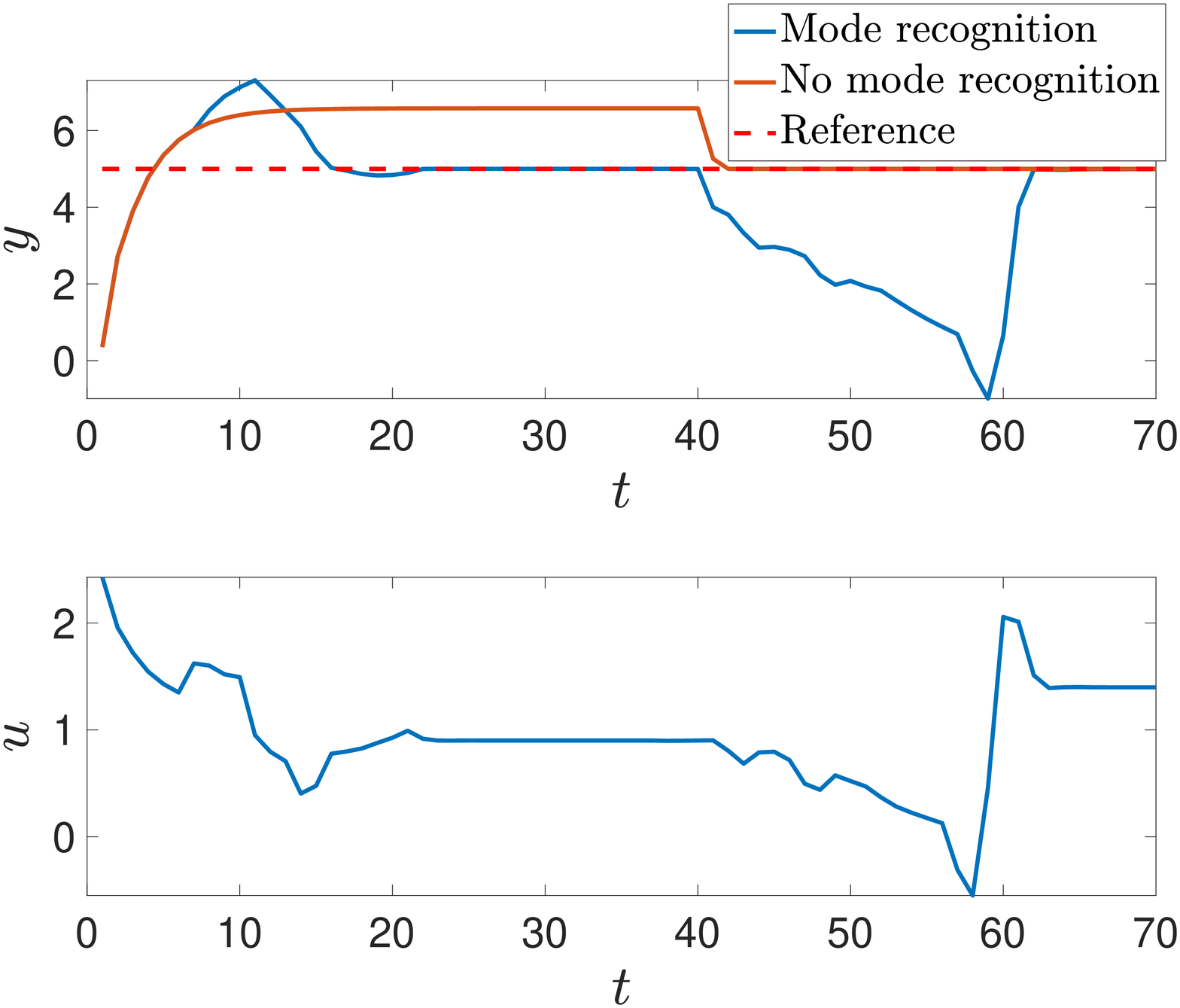}
\caption{Performance of mode recognition and control strategy on the SARX system with switches between modes compared to data-driven control without mode recognition.}
\label{fig:control_traj}
\end{figure}
\begin{figure}[h]
\includegraphics[width=1\linewidth]{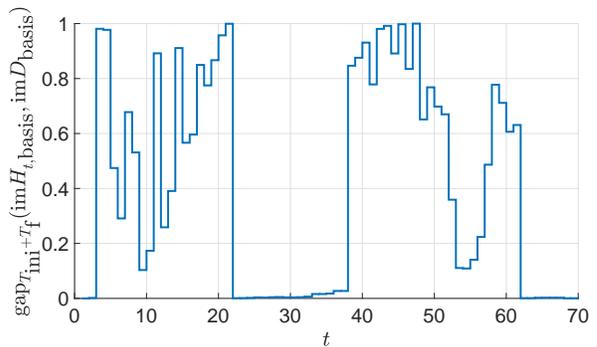}
\caption{Distance computed with the $L$-gap between $H_{t,\textup{basis}}$ and $D_{\textup{basis}}$ for $L=\Tini+\Tf$.}
\label{fig:gap}
\end{figure}

\section{Conclusion} \label{sec:conclusion}

This paper has explored the issue of uncertainty quantification in the behavioral setting. A new metric has been defined on the set of restricted behaviors and shown to capture parametric uncertainty for the class of  AR   models. The metric is a direct finite-time counterpart of the classical gap metric.  
A data-driven control case study has illustrated the value of the new metric through numerical simulations.

\begin{table*}
\centering
\caption{Metrics between subspaces $\mathcal{V}$ and $\mathcal{W}$ in $\Grass{k}{N}$  in terms of the corresponding principal angles 
$\{\theta_i\}_{i \in \mathbf{k}}$ and in terms of matrices $V$ and $W$ whose columns are
orthonormal bases for $\mathcal{V}$ and $\mathcal{W}$,   with ${V}^{\transpose} W = U  S Z^\transpose$ a (full) SVD.}
\begin{tabular}{lll} \hline
Metric     				& Principal angles 										 & Matrix representations  \\[0.3em] \hline
Asimov       			&   $d^{\alpha}(\mathcal{V}, \mathcal{W}) = \theta_k$           
							&   $\sin^{-1} \norm{VV^{\transpose}-WW^{\transpose}}_{2} $                \\
Binet-Cauchy  		&   $d^{\beta}(\mathcal{V}, \mathcal{W})= \left(1-\prod_{i=1}^k (\cos \theta_i)^2 \right)^{1/2}$          
							&   $\left(1- \det(V^{\transpose}W)^2\right)^{1/2}$    \\
Chordal      			&   $d^{\kappa}(\mathcal{V}, \mathcal{W})=  \left(\sum_{i=1}^k (\sin \theta_i)^2 \right)^{1/2}$          
							&   $\tfrac{1}{\sqrt 2} \norm{VV^{\transpose} - WW^{\transpose}}_F$                \\
Fubini-Study 		&   $d^{\phi}(\mathcal{V}, \mathcal{W})= \cos^{-1} \left(\sum_{i=1}^k (\cos \theta_i) \right)$          
							&   $\cos^{-1} |\det V^{\transpose}W| $                \\
Grassmann 			&   $d^{\gamma}(\mathcal{V}, \mathcal{W})= \left( \sum_{i=1}^k \theta_i^2 \right)^{1/2}$           
							&   $\norm{\cos^{-1} \Sigma}_{F}$                \\
Martin       			&   $d^{\mu}(\mathcal{V}, \mathcal{W})= \left(\log \left(\prod_{i=1}^k 1/(\cos \theta_i)^{2} \right) \right)^{1/2} $           
							&   $(-2\log\det V^{\transpose}W)^{1/2}$              \\
Procrustes   		&   $d^{\rho}(\mathcal{V}, \mathcal{W})= 2\left( \sum_{i=1}^k (\sin(\theta_i/2))^2 \right)^{1/2}$           
							&   $\norm{VU-WZ}_{F}$                 \\
Projection (gap) &   $d^{\pi}(\mathcal{V}, \mathcal{W})= \sin \theta_k$           
							&    $\norm{VV^{\transpose}-WW^{\transpose}}_{2}$                \\
Spectral     			&   $d^{\sigma}(\mathcal{V}, \mathcal{W})=2 \sin(\theta_k/2) $           
							&   $\norm{VU-WZ}_{2}$                 \\[0.3em] \hline
\end{tabular}
\label{tab:distances}
\end{table*}

The paper has shown that the gap induces a metric space structure on the set of restricted behaviors.  However,  there are many other common metrics defined on Grassmannians~\cite{deza2009encyclopedia}.  Table~\ref{tab:distances} recalls some of these metrics,  as well as formulae to compute them.   
The fact that all distances  in Table~\ref{tab:distances} depend on the principal angles is not a coincidence.  
In fact,  the geometry of the Grassmannian is such that any rotationally invariant metric between $k$-dimensional subspaces in $\R^N$  
(i.e., dependent only on the relative the position of subspaces) is necessarily a function of the principal angles.

\begin{theorem}\, \cite[Theorem\,2]{ye2016schubert} 
Let  ${d}$ be a rotationally invariant\footnote{A metric  ${d}$ on $\Grass{k}{N}$ is \textit{rotationally invariant} if
$$d(Q\cdot \mathcal{V}, Q\cdot \mathcal{W}) =  d(\mathcal{V}, \mathcal{W})$$
for all ${Q \in O(N)}$ and ${\mathcal{V}, \mathcal{W} \in \Grass{k}{N}}$~\cite[p.1179]{ye2016schubert},  where 
the \textit{left action} of the orthogonal group $\Ort{N}$ on $\Grass{k}{N}$ 
is defined as   
${Q\cdot \mathcal{V} = \Image (Q V)}$
for ${Q \in O(N)}$ and ${\mathcal{V} \in \Grass{k}{N}}$. 
} metric on $\Grass{k}{N}$.  
Then $d(\mathcal{V}, \mathcal{W})$ is a function of the principal angles $\{\theta_i\}_{i \in \mathbf{k}}$ 
between the subspaces $\mathcal{V}$ and $\mathcal{W}$.
\end{theorem}

\noindent
Each metric induces a particular geometry,  which comes with its own advantages and disadvantages. 
For example,  the \textit{Grassmann metric} is the geodesic distance on $\Grass{k}{N}$~\cite[Theorem\,2]{ye2016schubert},  viewed as a Riemannian (quotient) manifold with metric ${g_{Z}(V,W) = \trace((Z^{\transpose}Z)^{-1} V^{\transpose} W)}$.  The corresponding geodesics admit an explicit expression~\cite{absil2009optimization},  which 
allows for a number of optimization-based problems to be solved (e.g.,  regression~\cite{hong2014geodesic}).  
This,  in turn,  suggests that the choice of the metric structure on the set of restricted behaviors
is crucial,  raising a number of important questions.   
For instance,  any metric on $\Grass{k}{N}$ that induces a differentiable structure opens up the possibility of directly optimizing over behaviors.  So can one exploit any such structure to improve the performance of data-driven control algorithms (e.g.,  DeePC)? 
In practice,   non-parametric representations of restricted behaviors are typically constructed from noisy measurements.
This can be a serious drawback,  because noisy restricted behaviors may appear to be far apart,  even when close and/or of the same dimension.  
This issue may be elegantly resolved by extending the metrics in Table~\ref{tab:distances} to infinite Grassmannians~\cite{ye2016schubert}.  So can one leverage these results in an online,  real-time,  noisy setting where behaviors are constantly changing?  We leave the exploration of these important questions as future research directions.

\appendix 
\section{Background}

\subsection{The principal angles}  \label{ssec:appendix-angles}

Let ${\mathcal{V} \in \Grass{k}{N}}$ and ${\mathcal{W} \in \Grass{l}{N}}$.     Let ${r = \min(k, l)}$. 
For ${i \in \mathbf{r}}$,  the $i$-th \textit{principal vectors} $(p_i, q_i)$ are defined, recursively, as the solution of the optimization problem
\beq \nn \label{eq:optimization-problem-principal-angles}
\begin{array}{ll}
    \mbox{maximize}      & p^{\transpose}q\\
    \mbox{subject to}    & p \in \mathcal{V},   ~ p^{\transpose}p_1 = \ldots = p^{\transpose}p_{i-1} = 0,  \ \, \norm{p}_2 = 1 ,\\
    									& q \in \mathcal{W},   \ q^{\transpose}q_1 = \ldots = q^{\transpose}q_{i-1} = 0,  \ \, \norm{q}_2 = 1 .
\end{array}
\eeq 
The \textit{principal angles between the subspaces ${\mathcal{V}}$ and ${\mathcal{W} }$} are defined as
${\theta_{i} = \arccos(p_i^{\transpose}q_{i})}$ for ${i \in \mathbf{r}}$. Clearly,  ${0 \le \theta_1 \le \ldots  \le \theta_r \le  \tfrac{\pi}{2}}$.  Principal vectors and principal angles may be easily computed using, e.g., the SVD~\cite{golub2013matrix}.
Let ${V\in\R^{N\times k}}$ and ${W\in\R^{N\times l}}$ be a $k$-frame of ${\mathcal{V} \in \Grass{k}{N}}$ and an $l$-frame of ${\mathcal{W} \in \Grass{l}{N}}$, respectively.   
Let ${{V}^{\transpose} W = U  S Z^\transpose}$
be a (full) SVD of the matrix ${V}^{\transpose} W$, i.e. ,   ${U\in\Ort{k}}$,  ${V\in\Ort{l}}$, ${S = \blockdiag(S_1) \in \R^{k \times l}}$ with $S_1 = \diag(\sigma_1, \ldots, \sigma_r) \in \R^{r \times r}$, where $\sigma_1 \ge \ldots \ge \sigma_r \ge 0$. The principal angles can be computed as
${\theta_{i}  = \arccos(\sigma_i), }$ with ${i \in \mathbf{r}}$.   See~\cite{golub2013matrix} for further detail.

\subsection{Proof of Theorem~\ref{thm:gap_connection}} \label{ssec:proof_thm_2}

The proof of the theorem requires some preliminary results.  We first establish a one-to-one correspondence between $\R^{qL}$ and the subspace of $\R^q$-valued sequences with finitely many non-zero elements,  as well as additional elementary results (omitting the proof of those that may be established by direct computation).

Let $(\R^q)^{\infty}$ be the space of $\R^q$-valued sequences with finitely many non-zero elements,  i.e.,
$(\R^q)^{\infty} =
\{
w \in (\R^q)^{\Zge} \, : \,   \exists \, T\in \Zge \text{ s.t. }  {w_t = 0} , \, \forall \, t \ge T  
 \}. $
For ${L\in\N}$,   let ${\iota_L : \R^{qL} \to (\mathbb{R}^q)^{\infty}}$ be the \textit{inclusion map},  defined as
$\iota_L(w)=\left(w_{1},\ldots ,w_{L},0,0,\ldots \right).$
Then we have
$(\mathbb{R}^q)^{\infty} = \bigcup_{L=0}^{\infty}  \iota_L (\R^{qL}) .$
Thus,  each image $\iota_L (\R^{qL}) $ may be identified with $\mathbb{R}^{qL}$ by identifying each $ \left(w_{1},\ldots ,w_{L},0,0,\ldots \right) \in (\mathbb{R}^q)^{\infty}$ with ${w=\text{col}\left(w_{1},\ldots ,w_{L} \right) \in \mathbb{R}^{qL}}$.  
Furthermore,  the subspace topology on $\iota_L (\R^{qL}) $,  the quotient topology induced by the map $\iota_L $,  and the Euclidean topology on $\R^{qL}$,  all coincide.

\begin{lemma} \label{lemma:norm} 
Let ${w \in \R^{qL}}$. Then ${\iota_L(w) \in \ell_2^q}$ and 
${\norm{w}_2 = \norm{\iota_L(w)}_{\ell_2}}.$
\end{lemma}

\begin{lemma}  \label{lemma:lim_sup_inf} 
Let $\mathcal{S}$ be a subset of a topological space and let $f:\mathcal{S}\to \R$ be a continuous function. 
Assume ${\mathcal{S}_0 \subseteq \mathcal{S}_1  \subseteq \ldots }$ is a sequence of subsets of $\mathcal{S}$ such that 
${\mathcal{S} \subseteq \closure\left( \bigcup_{L=0}^{\infty} \mathcal{S}_L  \right) }$
then
$\sup_{\mathcal{S}} f  = \lim_{L \to \infty } \sup_{\mathcal{S}_L} f  ,$
whenever the limit exists.
\end{lemma}

\begin{proof}
The sequence $ \sup_{\mathcal{S}_L} f $ is non-decreasing,  so 
$\lim_{L \to \infty } \sup_{\mathcal{S}_L} f  = \sup_{L \in \Zge} \sup_{\mathcal{S}_L} f . $
Furthermore, by assumption,  ${\mathcal{S}_L \subseteq \mathcal{S}}$ for every ${L \in \Zge}$,  so 
$\sup_{\mathcal{S}} f  \ge  \sup_{L \in \Zge} \sup_{\mathcal{S}_L} f .$
To prove the reverse inequality,  consider a sequence ${w_L} \in \mathcal{S}$ such that ${f(w_L) \to \sup_{\mathcal{S}} f}$.
Then,  for every ${L\in \Zge}$ there is ${\epsilon_L>0}$ such that ${f(w) > f(w_L) - 1/L}$ for every ${w\in\mathcal{N}_{\epsilon_L}(w_L)}$, with $\mathcal{N}_{\epsilon_L}(w_L)$ an  $\epsilon_L$-neighborhood of $w_L$.  By assumption,  ${\mathcal{S} \subseteq \closure\left( \bigcup_{L=1}^{\infty} \mathcal{S}_L  \right)}$,  so every $\mathcal{N}_{\epsilon_L}(w_L)$ contains a point $\bar{w}_L \in \mathcal{S}_{k_L}$ for some $\mathcal{S}_{k_L}$.  Thus,  $f(\bar{w}_L) \to  \sup_{\mathcal{S}} f$.  This,  in turn,  implies
$\sup_{\mathcal{S}} f  \le  \sup_{L \in \Zge} \sup_{\mathcal{S}_L} f .$
\end{proof}

Given ${L \in \Zge}$,  let ${\Pi_L: (\R^q)^{} \to (\R^q)^{\Zge}}$ be the \textit{truncation operator},  defined as~\cite[p.13]{willems1971analysis}
\beq \nn
\Pi_L(w) = 
\begin{cases}
w_t,  & \text{ for } t \in [0,L-1] \cap \Zge,  \\
0, & \text{ otherwise},
\end{cases}
\eeq
with the convention that ${\Pi_0(w) = w}$ for all ${w\in(\R^q)^{\Zge}}$.

\begin{lemma} \label{lemma:inclusion_truncation}
Let $\B \in \L^{q}$.  Then  ${\iota_L(\B|_L) = \Pi_L(\B)}. $
\end{lemma}

\begin{lemma} \label{lemma:closure}
Let $\B \in  \L^{q} \cap \ell_2^q$.  Then
${\B = \closure\left( \bigcup_{L=0}^{\infty} \Pi_L (\B)  \right) }.$
\end{lemma}

\begin{proof}
($\subseteq$).  Since ${\B \cap  (\R^q)^{\infty} = \bigcup_{L=0}^{\infty} \Pi_L (\B) }$,
we have
\begin{align*}
\closure\left( \bigcup_{L=0}^{\infty} \Pi_L (\B)  \right)  
&= \closure\left( \B \cap  (\R^q)^{\infty} \right)  
\subseteq \closure(\B) \cap  \closure\left( (\R^q)^{\infty} \right) \\
&= \closure(\B) \cap  (\R^q)^{\Zge}
= \closure(\B) 
= \B,
\end{align*}
where we have used ${(\R^q)^{\Zge} = \closure\left( (\R^q)^{\infty} \right)}$~\cite[p151]{munkres2000topology} and
the completeness of $\B$.

($\supseteq$).  Let ${w\in\B}$ and recall that ${w \in \closure\left( \bigcup_{L=0}^{\infty} \Pi_L (\B)  \right)}$ if and only if 
${w}$ is the limit of some sequence of points in ${\bigcup_{L=0}^{\infty} \Pi_L (\B)}$.  Consider the sequence 
$\bar{w}_L = \Pi_L(w)$ for $L \in \Zge$.
Clearly,  ${\bar{w}_L \in {\bigcup_{L=0}^{\infty} \Pi_L (\B)}}$.  Furthermore,   since ${\B \in  \L^{q} \cap \ell_2^q}$,  ${\bar{w}_L \to w }$,  as desired. 
\end{proof}

\begin{proof}[Proof of Theorem~\ref{thm:gap_connection}]
By assumption,  $P$ and $\tilde{P}$ are bounded linear operators,  so $\text{gap}_{\ell_2}(P,\tilde{P})$ is well-defined.
Then
\begin{align}
\lim_{L \to \infty} \overset{\rightharpoonup}{\text{gap}}_L(\B,\tilde{\B}) 
& \stackrel{\eqref{eq:gapL_def}}{=} \lim_{L \to \infty} \overset{\rightharpoonup}{\text{gap}}(\B|_L,\tilde{\B}|_L)  \nn \\
& \stackrel{\eqref{eq:directed_gap}}{=}   \lim_{L \to \infty} \sup_{w \in \B|_{L} \atop \norm{w}_2 = 1} \inf_{\tilde{w} \in  \tilde{\B}|_{L}} \norm{w-\tilde{w}}_{2}  \nn  \\ 
& = \lim_{L \to \infty} \sup_{w \in \iota_L(\B|_{L}) \atop \norm{w}_{\ell_2} = 1} \inf_{\tilde{w} \in  \iota_L(\tilde{\B}|_{L})  } \norm{w-\tilde{w}}_{\ell_2}  \nn  \\ 
&  = \lim_{L \to \infty} \sup_{w \in \Pi_L(\B)) \atop \norm{w}_{\ell_2} = 1} \inf_{\tilde{w} \in \Pi_L(\tilde{\B})  } \norm{w-\tilde{w}}_{\ell_2}   \nn  \\  
&  \le  \lim_{L \to \infty} \sup_{w \in \Pi_L(\B)) \atop \norm{w}_{\ell_2} = 1} \inf_{\tilde{w} \in \tilde{\B}  } \norm{w-\tilde{w}}_{\ell_2}  \nn  \\  
&  =  \sup_{w \in \B \atop \norm{w}_{\ell_2} = 1} \inf_{\tilde{w} \in  \tilde{\B}}  \norm{w-\tilde{w}}_{\ell_2}    \nn  \\
&  = \overset{\rightharpoonup}{\text{gap}}_{\ell_2}(P,\tilde{P})  , \label{eq:proof_thm_connection} 
\end{align} 
where 
the third identity is a consequence of Lemma~\ref{lemma:norm},
the fourth identity follows 
from Lemma~\ref{lemma:inclusion_truncation}, 
the fifth inequality is implied by ${ \inf_{\tilde{w} \in \Pi_L(\tilde{\B})  } \norm{w-\tilde{w}}_{\ell_2}  \le  \inf_{\tilde{w} \in \tilde{\B}  } \norm{w-\tilde{w}}_{\ell_2}}$ for ${w \in \Pi_L(\B)}$,  the sixth  equality is a consequence of Lemma~\ref{lemma:lim_sup_inf},   and the last identity holds by Definition~\ref{def:gap_l2},  since ${\B = \text{graph}_{\ell_2} (P)}$ and ${\B = \text{graph}_{\ell_2} (\tilde{P})}$,  by assumption. 
Finally,  the function $(x,y) \mapsto \max(x,y)$ is continuous,   thus  
\begin{align*}
\lim_{L \to \infty} {\text{gap}}_L(\B,\tilde{\B}) 
&=
\lim_{L \to \infty} 
\max\{ 
\overset{\rightharpoonup}{\text{gap}}_L(\B,\tilde{\B}),
\overset{\rightharpoonup}{\text{gap}}_L(\tilde{\B},\B) \}  \\ 
&  \stackrel{\eqref{eq:proof_thm_connection}}{\le} 
\max\{ 
\overset{\rightharpoonup}{\text{gap}}_{\ell_2}(P,\tilde{P}) ,
\overset{\rightharpoonup}{\text{gap}}_{\ell_2}(\tilde{P},P) \}  \\  
&= {\text{gap}}_{\ell_2}(\tilde{P},P) .
\end{align*}
\end{proof}

\bibliographystyle{IEEEtran}
\bibliography{refs}

\end{document}